\documentclass[12pt]{amsart}
\usepackage{amsfonts}
\usepackage{amsmath}
\usepackage{amssymb}
\usepackage{graphicx}
\usepackage{amscd}
\usepackage{xcolor}
\usepackage{hyperref}
\usepackage{tikz-cd}
\usepackage{multirow}
\DeclareMathOperator\arctanh{arctanh}
\newtheorem{theorem}{Theorem}

\newtheorem{definition}[theorem]{Definition}
\newtheorem{example}[theorem]{Example}

\newtheorem{proposition}[theorem]{Proposition}

\newcommand{\ch}{\cosh}
\newcommand{\sh}{\sinh}

\begin{document}
	\author{G. Polychrou}
	\title[Functional separation solutions of the Sh-Gordon eqs] 
        {Functional separation solutions of the Sinh-Gordon type equations}
	
	\begin{abstract}
	In this paper, we address the following question: Which hyperbolic or elliptic 
        PDEs admit functional separable solutions. We shall focus on the study of a sinh-Gordon type equation. We construct solutions to this equation via the method of functional separation. We prove that these are the only families that have the property of functional separation and so we obtain a classification. To this end, we construct new families of solutions for the hyperbolic and elliptic versions of both sine and sinh Gordon equations in a unified way. 
	\end{abstract}
	\maketitle
	
	\section{Introduction and Statement of the Results}\label{Introduction}
    Consider the following PDE:
	\begin{equation}\label{Q}
		w_{xx} - \delta^2 w_{yy} = Q(w) , \quad \delta = \{1,i\}
	\end{equation}
	where $Q(w)$ is a smooth function. 
    For $\delta = 1$, we have a hyperbolic PDE while for $\delta = i$, an elliptic one.
   \par
    We define the functional separation property of solutions as follows: 
	\begin{definition}
		A solution $w(x,y)$ is functional separation if there exists an invertible, smooth function $f$ such that
		\[
		f\left(w(x,y)\right) = A(x) + B(y)
		\]  
		or equivalent the PDE is compatible with
		\[
		w_{xy} = F(w)w_xw_y,
		\]
		where $F(w) = -f''/f'$.
	\end{definition}
     It is natural to ask for which $Q$ we obtain functional separation solutions. This question was answered in \cite{GI,QHD}, for the case of the Dalamber (or wave) operator. For the sake of completeness, we shall present the proof.
	
    Moreover, in this paper, we study a sinh-Gordon type equation of the form
	\begin{equation}\label{GShGordon}
		\square_\delta w = \frac
		{2a}{\kappa}
		\sinh\left(2\kappa w\right),
	\end{equation}
	where $\square_\delta = \frac{\partial^2}{\partial x^2} - \delta^2 \frac{\partial^2}{\partial y^2}$, with $\delta,\kappa \in \{1,i\}$ and $a \in \mathbb{R}$. 
	The $\square_\delta$ operator is either the wave operator for $\delta = 1$ or the Laplace operator for $\delta = i$.
	
	Equation (\ref{GShGordon}), for a suitable choice of $\delta$ and $\kappa$, turns out to be the hyperbolic or elliptic version of the sinh-Gordon equation or sine-Gordon equation. More precisely, the following table presents the equations we have for every choice of $\delta$ and $\kappa$:
	\begin{center}
		\begin{tabular}{ |c|c|c| } 
			\hline
			$\delta$ & $\kappa$ & Equation \\
			\hline
			1 & 1 &  $w_{xx} - w_{yy} = 2a\sh(2w)$\\ 
			\hline
			1 & i & $w_{xx} - w_{yy} = 2a\sin(2w)$ \\ 
			\hline
			i & 1 & $w_{xx} + w_{yy} = 2a\sinh(2w)$ \\ 
			\hline
			i & i & $w_{xx} + w_{yy} = 2a\sin(2w)$ \\
			\hline
		\end{tabular}
	\end{center}
	By studying equation (\ref{GShGordon}), one studies in a unified way all the equations of the table. In this work, we construct solutions of equation (\ref{GShGordon}). There is no difference between the hyperbolic and elliptic versions and the sine-Gordon and sinh-Gordon, and we study them in a unified way.
	
	These equations are important and appear in numerous applications, for example, \cite{AO-YAN, E.R., FokPell, Hu-Sun, Hwang, McKean, Wei, Rubinstein, YangZhong, Wazwaz}. Research interest in this type of equation increased scientifically since the introduction of the soliton theory. For more details, we refer to \cite{AblowitzSegur,Hirota1, Hirota2,  Miura, RogersSchief, WazwazBook} and the references therein. The elliptic sinh-Gordon equation plays a vital role in the theory of constant mean curvature surfaces \cite{Abresch, Bour, Joaquin, K, SterlingWente, Wente} and in the theory of harmonic maps \cite{FordyWood, FotDask1, FotDask2, Minsky, Pol, PPFD2021, Wolf}. In the recent work of Fotiadis and Daskaloyannis, \cite{FotDask2}, equation (\ref{GShGordon}) plays a vital role in the construction of harmonic diffeomorphisms between pseudo-Riemannian surfaces. 
	
	Moreover, sine-Gordon equations have several applications in Physics and Engineering. For instance, in \cite{Hirota2, YangZhong} and the references therein, the motion of the magnetic flux on a one-dimensional Josephson junction transmission line is described by the hyperbolic sine-Gordon. Other applications of the sine-Gordon include the distributed mechanical analogue transmission line, the motion of a slide dislocation in a crystalline structure, models of elementary particles, the transmission of ferromagnetic waves, the epitaxial growth of thin films, the DNA-soliton dynamics. Finally, the flux of fluid across a closed curve is quantized in the case of the elliptic sine-Gordon equation, \cite{Kaptsov}.
	
    \par In Section 2, we prove the following Proposition.
    \begin{proposition}\label{Q forms}
		The equation (\ref{Q}) admits functionally separation solution if $Q(w)$ has the following form
		\begin{enumerate}
			\item $Q(w) = c_1 w + c_2,$
			\item $Q(w) = c_1e^{2w} + c_2e^{-w},$
			\item $Q(w) = c_1 e^{2wu} + c_2 e^{-cw},$
			\item $Q(w) = c_1 w + c_2 w\log(w),$
			\item $Q(w) = c_1\sh(2w) + c_2
			\left(
			2\cosh(w) + \arctan(\sh(w))
			\right),$
			\item $Q(w) = c_1\sin(2w) + \frac{c_2}{4}
			\left(
			\sin(2w)\arctanh(\sin(w)) - 2\cos(w)
			\right),$
			\item  $Q(w) =  c_1\sin(2w) + c_2
			\left(
			2\sin(w) + \sin(2w)\log\left(\tan(\frac{w}{2})\right)
			\right),$
			\item $Q(w) = c_1\sh(2w) + c_2
			\left(
			2\sh(w) + \sh(2w)\log\left(\tanh(\frac{w}{2})\right)
			\right).$
		\end{enumerate} 
	\end{proposition}
    
	\par Next, we study (\ref{GShGordon}). In Section 3, we prove the following main results. Recall that $\kappa = \{1,i\}$.
	\begin{proposition}\label{F(w)}
		Let $w(x,y)$ be a solution to (\ref{GShGordon}) and $w$ is 
		compatible with 
		\begin{equation}\label{compatible}
			w_{xy} = F(w)w_xw_y.
		\end{equation}
		Then $F(w)$ is of the form
		\begin{equation}\label{1st Fam of GSh-G}
			F(w) = \kappa \tanh\left(\kappa w\right), 
		\end{equation}
		or of the form
		\begin{equation}\label{2nd Fam of GSh-G}
			F(w) = \kappa \coth\left(\kappa w\right).
		\end{equation}
	\end{proposition}
	\begin{theorem}\label{Family of solutions}
		The equation (\ref{GShGordon}) admits functional separation solutions 
		\[
		\sh\left(\kappa w(x,y)\right)
		= \tan\left(\kappa
		\left( A(x) + B(y) \right)\right),
		\] 
		and
		\[
		\tanh\left(\kappa \frac{w(x,y)}{2}\right) =
		\kappa
		e^{-2\left(A(x) + B(y)\right)}.
		\]
	\end{theorem}
	
	The Theorem above demonstrates the classification of solutions of functional separation for equation (\ref{GShGordon}). Since we have not selected the values of $\delta$ and $\kappa$, there is no distinction between the elliptic and hyperbolic forms.
	
	Note that our results extend the results in \cite{Pol, PPFD2021}, since in these papers, the authors construct solutions for the elliptic sine-Gordon and elliptic sinh-Gordon equation. In this work, we study it in a unified way.
	
	\par In Section 4, we consider the solutions constructed in Theorem \ref{Family of solutions}, to obtain the form of the functions $A(x)$ and $B(y)$. It turns out that these functions are Jacobi elliptic.
	
	At first, we study the solutions of the form
	\begin{equation}\label{1stFamily}
		\tanh
		\left(
		\kappa \frac{w(x,y)}{2}
		\right)
		= \kappa
		e^{-2\left(A(x) + B(y)\right)} 
		= \kappa F(x)G(y),
	\end{equation}
	and we state the following result.
	\begin{theorem}\label{1stFam}
		If $w$ is a solution of the  sinh-Gordon type equation 
		(\ref{GShGordon})
		of the form
		\[
		\tanh
		\left(
		\kappa\frac{w(x,y)}{2}
		\right)
		= \kappa F(x)G(y),
		\]
		then the functions $F(x)$ and $G(y)$ satisfy the differential equations
		\begin{equation}\label{F}
			F'(x)^{2} = A\kappa^2\delta^2F(x)^{4} + B
			F(x)^{2} + C,
		\end{equation}
		\begin{equation}\label{G}
			G'(y)^{2} = 
			\kappa\delta^2 CG(y)^{4} + 
			\delta^2
			\left(
			B - 4a
			\right)
			G(y)^{2} + A,
		\end{equation} 
		where $A,B$ and $C \in \mathbb{R}$.
	\end{theorem} 
	
	Functions $F$ and $G$ are in general Jacobi elliptic functions, but there are constants for which the solutions are elementary functions.
	
	The second family of solutions $w$ of the generalised sinh-Gordon equation is of the form
	\begin{equation}\label{2ndFamily}
		\sh\left(
		\kappa w(x,y)
		\right)
		= \tan
		\left(
		\kappa
		(A(x) + B(y))
		\right),
	\end{equation}
	and we obtain the following result.
	\begin{theorem}\label{2ndFam}
		If $w$ is a solution of the  sinh-Gordon type equation 
		(\ref{GShGordon})
		of the form 
		\[
		\sh\left(
		\kappa w(x,y)
		\right)
		= \tan
		\left(
		\kappa
		(A(x) + B(y))
		\right),
		\]
		then the functions $A(x)$ and $B(y)$ satisfy the differential equations
		\begin{equation}\label{A}
			A''(x)^2 = - \kappa^2A'(x)^4 + c_1 A'(x)^2 + c_2
		\end{equation}
		\begin{equation}\label{B}
			B''(y)^{2} = - \kappa^2 B'(y)^{4} + \delta^2\left(c_1 - 8a\right)B'(y)^{2} + c_3,
		\end{equation} 
		where $16a^2 + \kappa^2\left(c_3 - c_2\right) = 4ac_1$ and $c_1,c_2, c_3 \in \mathbb{R}$.
	\end{theorem}
	\textbf{Remark:} By letting $\alpha(x) = A'(x)$ and $\beta(y) = B'(y)$, the equations (\ref{A}) and (\ref{B}) turn into
	\[
	\alpha'(x)^2 = -\kappa^2 \alpha^4(x) + c_1\alpha^2(x) + c_2,
	\]
	\[
	\beta'(y)^2 = - \kappa^2 \beta(y)^{4} + \delta^2\left(c_1 - 8a\right)\beta(y)^{2} + c_3,
	\]
	where $16a^2 + \kappa^2\left(c_3 - c_2\right) = 4ac_1$ and $c_1,c_2, c_3 \in \mathbb{R}$.
	
	Generally, the solutions of these equations are Jacobi elliptic functions, but there are constants $c_1,c_2$ and $c_3$, for which we obtain elementary functions.
	
    \section{Proof of Proposition \ref{Q forms}}

	\begin{proof}\label{Proof}
		The proof follows \cite{Arrigo, GI, QHD}. The equations $w_{xx} - \delta^2 w_{yy} = Q(w)$ and $w_{xy} = F(w)w_xw_y$ have to be compatible, to admit functional separation solutions. Differentiating both equations, once with respect to $x$ and once with respect to $y$, we obtain the following:
		\begin{align*}
			w_{xxx} &= \delta^2 \left(F'(w) + F^2(w)\right)w_xw_y^2 + 
			\delta^2F(w)w_{x}w_{yy} + Q'(w)w_x
			\\
			w_{yyy} & = \delta^2 \left(F'(w) + F^2(w)\right)w_x^2w_y
			+ \delta^2 F(w)w_yw_{xx} - \delta^2Q'(w)w_y
			\\
			w_{xyx} &= F'(w)w_x^2w_y + F(w)w_{xx}w_y + F^2(w)w_x^2w_y,
			\\
			w_{xyy} &= F'(w)w_xw_y^2 + F^2(w)w_xw_y^2 + F(w)w_xw_{yy}.
		\end{align*}
		Differentiating and using (\ref{Q}) to eliminate the second derivatives, we obtain
		\begin{align*}
			\frac{w_{xxyx}}{w_xw_y} &= 
			\left(3F'(w) + 3F^2(w)\right)Q(w) + F(w)Q'(w)
			\\
			&+ \left(F''(w) + 3F(w)F'(w) + F^3(w)\right)w_x^2
			\\
			&+ \delta^2 \left(F'(w)F(w) + F^3(w)\right)w_y^2
			\\
			&+ \delta^2\left(3F'(w) + 4F^2(w)\right)w_{yy},
		\end{align*}
		\begin{align*}
			\frac{w_{xxxy}}{w_xw_y} & = Q''(w) + F^2(w)Q(w)
			\\
			&+ \left(F'(w)F(w) + F^3(w)\right)w_x^2
			\\
			&+ \delta^2 \left(F''(w) + 3F(w)F'(w) + F^3(w)\right)w_y^2
			\\
			&+ \delta^2\left(3F'(w) + 4F^2(w)\right)w_{yy}.
		\end{align*}
		Using the compatibility of $w_{xxyx}$ and $w_{xxxy}$ and equate the terms, we obtain:
		\begin{align}\label{FandQ}
			3Q(w)F'(w) + 2 Q(w)F^2(w) + Q'(w)F(w) = Q''(w)
		\end{align}
            \begin{align}\label{EquationF}
                F''(w) + 2F'(w)F(w) = 0
            \end{align}
		Using (\ref{EquationF}), we obtain that $F$ is one of the following:
		\begin{enumerate}
			\item  $F(w) = c$, where $c$ is a real constant.
			\item  $F(w) = 1/w$
			\item  $F(w) = - \tan(w)$
			\item  $F(w) = \tanh(w)$
			\item  $F(w) = \cot(w)$
			\item  $F(w) = \coth(w)$
		\end{enumerate}
		We now find $Q$. From now on, $c_1$ and $c_2$ are real constants.
		\\
		\textbf{Case 1.} $F(w) = 0, Q''(w) = 0$ so 
		\[
		Q(w) = c_1w + c_2.
		\]
		\textbf{Case 2.} $F(w) = 1, Q''(w) - Q'(w) - 2Q(w) = 0$ so 
		\[
		Q(w) = c_1 e^{-w} + c_2 e^{2w}.
		\]
		\\
		\textbf{Case 3.} $F(w) = c, Q''(w) - cQ'(w) - 2c^2Q(w) = 0$ so 
		\[
		Q(w) = c_1e^{2cw} + c_2e^{-cw}.
		\]
		(For $c = -1$ \textbf{Bullough-Dodd-Mikhailov equation})
		\\
		(For $c = - 1$ and $c_1 = 0$ \textbf{Liouville})
		\\
		\textbf{Case 4.} $F(w) = \frac{1}{w}, w^2Q''(w) - wQ'(w) + Q(w) = 0$ so
		\[
		Q(w) = c_1w + c_2w\log(w).
		\]
		\textbf{Case 5.} $F(w) = \tanh(w)$, 
		\\
		$Q''(w) - \tanh(w)Q'(w) - 
		\left(\sec h^2(w) + 2\right)Q(w) = 0.$
		\\
		An obvious solution is $\sh(2w)$. So let $Q(w) = \sh(2w)R(w)$ the equation turns into 
		\[
		\frac{d}{dw}\left(\ch(w)\sh^2(w)R'(w)\right) = 0 ,
		\]
		\[
		R(w) = c_2\left(\frac{1}{\sh(w)} + \arctan(\sh(w))\right) + c_1,
		\]
		\[
		\Rightarrow Q(w) = c_1\sh(2w) + c_2
		\left(
		2\cosh(w) + \arctan(\sh(w))
		\right)
		\]
		\\
		\textbf{Case 6.} $F(w) = -\tan(w)$, 
		\\
		$Q''(w) + \tan(w)Q'(w) + 
		\left(\sec^2(w) + 2\right)Q(w) = 0,$
		\\
		An obvious solution is $\sin(2w)$. So let $Q(w) = \sin(2w)R(w)$ the equation turns into 
		\[
		\frac{d}{dw}
		\left(
		4\cos(w)\sin^2(w)R'(w) = 0,
		\right)
		\]
		\[
		R(w) = c_1 + \frac{c_2}{4}
		\left(\arctanh(\sin(w)) - \csc(w)\right),
		\]
		\[
		\Rightarrow Q(w) = c_1\sin(2w) + \frac{c_2}{4}
		\left(
		\sin(2w)\arctanh(\sin(w)) - 2\cos(w)
		\right)
		\]
		\textbf{Case 7.} $F(w) = \cot(w)$, 
		\\
		$Q''(w) - \cot(w)Q'(w) + \left(\csc^2(w) + 2\right)Q(w) = 0,$
		\\
		An obvious solution is $\sin(2w)$. So let $Q(w) = \sin(2w)R(w)$ the equation turns into
		\[
		\frac{d}{dw}
		\left(
		\cos^2(w)\sin(w)R'(w)
		\right)
		= 0,
		\]
		\[
		R(w) = c_1 + c_2
		\left(
		\sec(w) + \log\left(\tan(\frac{w}{2})\right),
		\right)
		\]
		\[
		\Rightarrow Q(w) = c_1\sin(2w) + c_2
		\left(
		2\sin(w) + \sin(2w)\log\left(\tan(\frac{w}{2})\right)
		\right)
		\]
		\textbf{Case 8.} $F(w) = \coth(w)$, 
		\\
		$Q''(w) - \coth(w)Q'(w) +
		\left(\csc h^2(w) - 2\right)Q(w) = 0$.
		\\
		An obvious solution is $\sh(2w)$. So let $Q(w) = \sh(2w)R(w)$ the equation turns into 
		\[
		\frac{d}{dw}\left(\ch^2(w)\sh(w)R'(w)\right) = 0 ,
		\] 
		\[
		R(w) = c_1 + c_2\left(\frac{1}{\ch(w)} + \log\left(\tanh(\frac{w}{2})\right)\right),
		\]
		\[
		\Rightarrow Q(w) = c_1\sh(2w) + c_2
		\left(
		2\sh(w) + \sh(2w)\log\left(\tanh(\frac{w}{2})\right).
		\right)
		\]
	\end{proof}

	\section{Solutions via functional separation}
	
	In this section, we demonstrate the proofs of Proposition \ref{F(w)} and of Theorem \ref{Family of solutions}. Through this Theorem, we classify the solutions with the property of functional separation of equation (\ref{GShGordon}). There is no distinction between the elliptic and hyperbolic versions of both the sine-Gordon and sinh-Gordon equations.
	
	Such solutions are important and many researchers have studied them. A remarkable example is due to Abresch.
	In \cite{Abresch}, the author describes the Wente Torus, introduced by Wente in \cite{Wente}. The author constructs families of solutions of the elliptic sinh-Gordon equation which satisfy a geometric condition. This condition turns out to be a nonlinear PDE of the form
	\[
	w_{xy} = \tanh(w)w_xw_y,
	\]
	or
	\[
	w_{xy} = \coth(w) w_xw_y.
	\]
	In this work, we prove that these conditions are related with the functional separation families of solutions. We remind that by Definition 1, a solution has the property of functional separation, if there exists a $f$ such that
	\[
	f\left(w(x,y)\right) = A(x) + B(y)
	\]
	or equivalent the equation (\ref{GShGordon}) is compatible with
	\[
	w_{xy} =  F(w) w_xw_y ,
	\]
	where $F(w) : \mathbb{R} \rightarrow \mathbb{R}$, $F = -f''/f'$. Our first goal is to prove Proposition \ref{F(w)}, to find explicitly the function $F$.
	\begin{proof}[Proof of Proposition (\ref{F(w)})]
		The equations (\ref{GSh-Gordon})
		and $w_{xy} = F(w)w_xw_y$ have to be compatible. Taking the derivatives with respect to $x$ and $y$  of both equations  we obtain
		\begin{align}
			w_{xxx} - \delta^2 w_{yyx} &= 
			4a w_x\ch\left(2\kappa w\right),
			\label{wxxx}\\
			w_{xxy} - \delta^2 w_{yyy} &=
			4a w_y\ch\left(2\kappa w\right),
			\label{wyyy}\\
			(w_{xy})_x &= F'(w)w_x^2w_y + F(w)w_{xx}w_y + F^2(w)w_x^2w_y,
			\label{wxyx} \\
			(w_{xy})_y &= F'(w)w_xw_y^2 + F^2(w)w_xw_y^2 + F(w)w_xw_{yy}
			\label{wxyy}.
		\end{align}
		Solving these equations for the third derivatives 
		$w_{xxx},w_{xxy},w_{xyy}$ and $w_{yyy}$ we obtain
		\begin{align*}
			w_{xxx} &= \delta^2 w_xw_y^2 
			\left(F'(w) + F^2(w)\right)
			+ \delta^2 F(w)w_xw_{yy} 
			+ 4a w_x \ch\left(2\kappa w\right),
			\\
			w_{xxy} &= F'(w)w_x^2w_y + F(w)w_{xx}w_y + F^2(w)w_x^2w_y,
			\\
			w_{xyy} &= F'(w)w_xw_y^2 + F^2(w)w_xw_y^2 + F(w)w_xw_{yy},
			\\
			w_{yyy} &= \delta^2w_x^2w_y
			\left(F'(w) + F^2(w)\right)
			+ \delta^2 F(w)w_{xx}w_y 
			- 4a\delta^2  w_y \ch\left(2\kappa w\right).
		\end{align*}
		By compatibility and using (\ref{GShGordon}) to eliminate the second derivatives, we obtain
		\begin{align*}
			\frac{w_{xxxy}}{w_xw_y} &=
			4aF(w)
			\ch\left(2\kappa w\right)
			+ \frac{6a}{\kappa}
			\left(F'(w) + F^2(w)\right)
			\sh\left(2\kappa w\right)
			\\
			&+ \left(
			F''(w) + 3F'(w)F(w) + F^3(w)
			\right)w_x^2\\
			&+ \left(\delta^2F(w)F'(w) + \delta^2F^3(w)\right)w_y^2\\
			&+ \left(
			3\delta^2F'(w) + 4\delta^2F^2(w)
			\right)w_{yy}
		\end{align*}
		and 
		\begin{align*}
			\frac{w_{xxyx}}{w_xw_y} &=
			8a \kappa
			\sh\left(2\kappa w\right)
			+ \frac{2a}{\kappa}F^2(w)\sh\left(2\kappa w\right)\\
			&+ \left(F(w)F'(w) + F^3(w)\right)w_x^2\\
			&+ \left(\delta^2F''(w) + 3\delta^2F'(w)F(w) + \delta^2F^3(w)\right)w_y^2\\
			&+ \left(3\delta^2F'(w) + 4\delta^2F^2(w)
			\right)w_{yy},
		\end{align*}
		
		Thus,
		\begin{align*}
			&4 F(w)\ch\left(2\kappa w\right) +
			\frac{4}{\kappa}F^2(w)
			\sh\left(2\kappa w\right) + \frac{6}{\kappa}F'(w)
			\sh\left(2\kappa w\right)
			- 8\kappa\sh\left(2\kappa w\right)
			\\
			& + \left(F''(w) + 2F'(w)F(w)\right)
			\left(w_x^2 - \delta^2w_y^2\right) = 0,
		\end{align*}
		from which we obtain the equations for $F(w)$
		\begin{equation}
			2\kappa F(w) \coth(2\kappa w)+
			3F'(w) + 
			2F^2(w)
			= 4\kappa^2,
			\label{1st equation of F}
		\end{equation}
		and
		\begin{align}
			F''(w) + 2F'(w)F(w) = 0.
			\label{2nd equation of F}
		\end{align}
		We begin with the equation 
		(\ref{2nd equation of F}) and so by integrating we obtain
		$F'(w) + F^2(w) = c$. We have to consider three cases for $c$. Namely $c = 0$, $c > 0$ and $c < 0$. By scaling, we consider $c = \{0,1,-1\}$. 
		The only solutions of (\ref{2nd equation of F}) compatible with (\ref{1st equation of F}) are
		$F(w) = \kappa \tanh (\kappa w)$ and
		$F(w) = \kappa \coth(\kappa w)$.
	\end{proof}
	Since we obtained the function $F$, we can construct the families of solutions via functional separation as in Theorem \ref{Family of solutions}. 
	\begin{proof}[Proof of Theorem \ref{Family of solutions}]
		By Proposition \ref{F(w)} we have that
		$F(w) = \kappa \tanh (\kappa w)$ or 
		$F(w) = \kappa \coth(\kappa w)$. To construct the families of solutions, we have to obtain the function $f$. Taking into account that $F = - f''/f'$ and considering the first case, we obtain
		\[
		-\frac{f''}{f'} = \kappa\tanh(\kappa w),
		\]
		and by integration we obtain
		\[
		f'(w) = 1/\cosh w ,
		\]
		which corresponds to (\ref{1st Fam of GSh-G}), while taking into account the second case
		\[
		-\frac{f''}{f'} = \kappa\coth(\kappa w),
		\]
		by integrations we obtain (\ref{2nd Fam of GSh-G}).
	\end{proof}
	
	\section{Solutions of the Generalised Sinh-Gordon equation}\label{Generalised Sinh-Gordon}
	In this section, we prove Theorems \ref{1stFam} and \ref{2ndFam}. These Theorems describe the families of solutions of (\ref{GShGordon}), which are classified by Theorem \ref{Family of solutions}. 
	
	The first family of solutions for (\ref{GShGordon}) are of the form
	\begin{equation}\label{FandG}
		\tanh
		\left(
		\kappa\frac{w(x,y)}{2}
		\right)
		= \kappa F(x)G(y).
	\end{equation}
	\begin{proof}[Proof of Theorem \ref{1stFam}]
		Let $H(y) = 1/G(y)$. By substituting in (\ref{GShGordon}), (\ref{FandG}) then we obtain
		\[
		\left(
		H^2(y) - \kappa^2 F^2(x)
		\right)
		\left(
		F''(x)H(y) + \delta^2 F(x)H''(y)
		\right)
		+
		2\kappa^2 F(x)F'(x)^2H(y)
		\]
		\[
		- 2\delta^2 F(x)H(y)H'(y)^2
		=
		4a
		\left(
		\kappa^2 F^3(x)H(y) + F(x)H^3(y)
		\right).
		\]
		By dividing with $F(x)H(y)$ we have
		\begin{equation}\label{FandH}
			\left(\frac{F''(x)}{F(x)} + \delta^2 \frac{H''(y)}{H(y)} \right)
			\left(
			H^2(y) - \kappa^2 F^2(x)
			\right)
			+ 2\kappa^2F'(x)^2 - 2\delta^2H'(y)^2 
		\end{equation}
		\[
		= 4a
		(\kappa^2F^2(x) + H^2(y)).
		\]
		By differentiating with respect to x and y, we obtain
		\begin{equation}\label{FandHconstant}
			\left(\frac{F''(x)}{F(x)}\right)_{x}
			\frac{1}{\kappa^2\delta^2F(x)F'(x)} = 
			\left(\frac{H''(y)}{H(y)}\right)_{y}
			\frac{1}{H(y)H'(y)} = 4A.
		\end{equation}
		So, we obtain
		\[
		\left(\frac{F''(x)}{F(x)}\right)_{x} = 4\kappa^2\delta^2AF(x)F'(x)
		\]
		\[
		\Rightarrow
		(F'(x))^{2} = A\kappa^2\delta^2
		(F(x))^{4} + B(F(x))^{2} + C,
		\]
		and for $H(y)$ we obtain
		\[
		\left(\frac{H''(y)}{H(y)}\right)_{y} = 4AH(y)H'(y)
		\]
		\[
		\Rightarrow 
		(H'(y))^2 = A(H(y))^4 + B'(H(y))^2 + C'.
		\]
		Using (\ref{FandH}) we obtain that $B - \delta^2 B' = 4a$ and $C = \kappa^2\delta^2C'$ and  equations (\ref{F}) and (\ref{G}) follow.
	\end{proof}
	By this Theorem, we obtain the exact form of solutions of the family (\ref{FandG}). In general, $F$ and $G$ are Jacobi elliptic functions. Through this Theorem, we construct new families of solutions for both the elliptic and hyperbolic versions, since there is no selection of $\delta$ and $\kappa$.
	Let us demonstrate an example of this family.
	\begin{figure}\caption{}
		\centering
		\includegraphics[scale = 0.6]{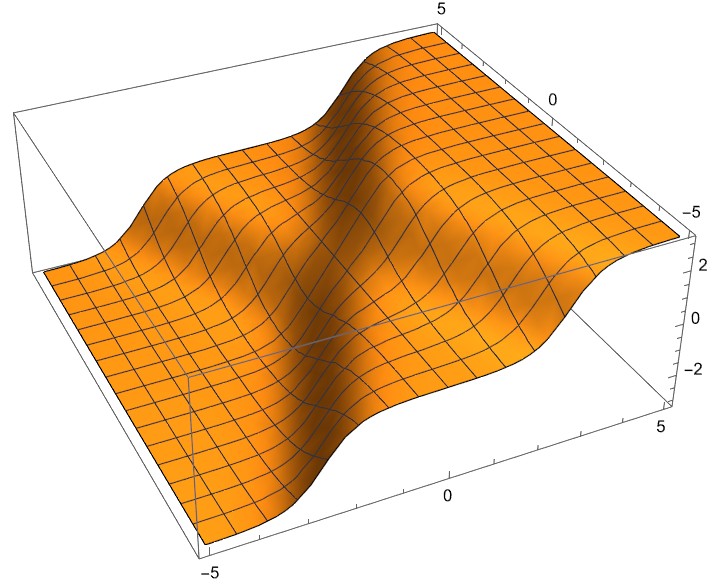}
	\end{figure}
	\begin{example}
		Let $a = 1$, $A = 0$, $B = 4 - \kappa^2$ and $C = 1$ then we obtain the equations
		\[
		F'(x)^2 = 
		\left(4 - \kappa^2\right)F^2(x) + 1,
		\quad
		H'(y)^2 = \kappa^2\delta^2
		\left(1 - H^2(y)\right).
		\]
		Then 
		\[
		F(x) = \frac{\sinh\left(\sqrt{4 - \kappa^2} x\right)}
		{\sqrt{4 - \kappa^2}},
		\quad
		H(y) = \cos\left(\kappa\delta y\right).
		\]
		So, for the equation (\ref{GShGordon}),
		we have constructed the solution
		\[
		\tanh\left(\frac{\kappa w(x,y)}{2}\right)
		= \kappa
		\frac{\sinh\left(\sqrt{4 - \kappa^2} x\right) \sec(\kappa\delta y)}
		{\sqrt{4 - \kappa^2}}.
		\]
		
		By selecting the constants $\delta$ and $\epsilon$, one obtains solutions for the hyperbolic or elliptic versions, depending on the choice made. For instance let  $\delta = 1$ and $\kappa = i$ then we obtain a solution of the hyperbolic sine-Gordon equation
		\[
		w(x,y) = 2\arctan
		\left(
		\frac{\sh(\sqrt{5}x)}{\sqrt{5}\ch(y)}
		\right).
		\]
		In Figure 1, we demonstrate its graph.
	
		Another example is to select $\delta = i$  and $\kappa = 1$ to admit a solution of the elliptic sinh-Gordon equation. To this end, we obtain the solution
		\[
		\tanh
		\left(
		\frac{w(x,y)}{2}
		\right)
		= \frac{\sinh(\sqrt{3}x)}{\sqrt{3}\ch(y)},
		\]
		which is well-defined in 
		$\Omega:= \{ (x,y) \in \mathbb{R}^2: 
		|\frac{\sinh(\sqrt{3}x)}{\sqrt{3}\ch(y)} | < 1\}$.
		In Figure 2, we demonstrate its graph.
		\begin{figure}\caption{}
			\centering
			\includegraphics[scale = 0.6]{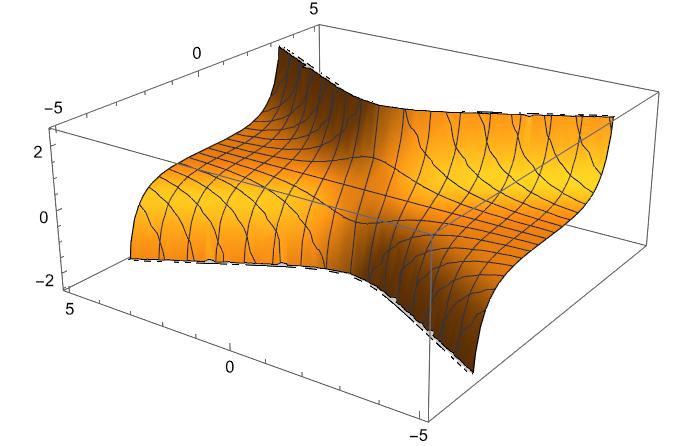}
		\end{figure}
		One can verify the results either by hand or using computing software, such as Mathematica.
	\end{example}
	The second family of solutions for (\ref{GShGordon}) are of the form
	\begin{equation}\label{AandB}
		\sh
		\left(
		\kappa w(x,y)
		\right)
		= \tan
		\left(
		\kappa(A(x) + B(y))
		\right)
		.
	\end{equation}
	\begin{proof}[Proof of Theorem \ref{2ndFam}]
		By substituting (\ref{AandB}) in (\ref{GShGordon}),  and dividing by \\ $\sec\left(\kappa(A(x) + B(y))\right)$ then we obtain
		\[
		\tan
		\left(
		\kappa(A(x) + B(y))
		\right)
		= \frac
		{
			\kappa
			\left(A''(x) - \delta^2 B''(y)\right)
		}
		{
			4a - \kappa^2A'(x)^2 + \delta^2\kappa^2 B'(y)^2 
		}.
		\]
		By differentiating with respect to x  and using the trigonometric identity $\sec^2u = 1 + \tan^2u$ and dividing with $A'(x)$, we obtain
		\begin{equation}\label{A''}
			\left(4a - \kappa^2A'(x)^2 + \delta^2\kappa^2 B'(y)^2\right)^2 
			+ \kappa^2 B''(y)^2
		\end{equation}
		\[
		=  \frac{A^{(3)}(x)}{A'(x)}
		\left(4a - \kappa^2A'(x)^2 + \delta^2\kappa^2 B'(y)^2\right)
		+ \kappa^2 A''(x)^2.
		\]
		By differentiating with respect to y we obtain
		\[
		4\delta^2\kappa^2A'(x)B'(y)B''(y)
		\left(4a - \kappa^2A'(x)^2 + \delta^2\kappa^2 B'(y)^2\right) 
		\]
		\[
		= 2\delta^2\kappa^2A^{(3)}(x)B'(y)B''(y) 
		- 2\kappa^2A'(x)B''(y)B^{(3)}(y)
		.
		\]
		By calculations, we obtain
		\[
		\frac{A^{(3)}(x)}{A'(x)} + 2\kappa^2A'(x)^2
		= \delta^2\frac{B^{(3)}(y)}{B'(y)} + 2\delta^2\kappa^2B'(y)^2 + 8a = c_1 
		\]
		So we have for $A(x)$
		\[
		\frac{A'''(x)}{A'(x)} + 
		2\kappa^2 A'(x)^2 = c_1
		\]
		\[
		\Rightarrow
		A''(x)^{2} = -\kappa^2A'(x)^{4} 
		+ c_1A'(x)^{2} + c_2
		\]
		and we have for $B(y)$
		\[
		\frac{B'''(y)}{B'(y)}
		+ 2\kappa^2B'(y)^2 
		= \delta^2
		\left(
		c_1 - 8a
		\right)
		\]
		\[
		\Rightarrow
		B''(y)^{2} = 
		- \kappa^2B'(y)^{4} + 
		\delta^2
		\left(
		c_1 - 8a
		\right)
		B'(y)^{2} + c_3.
		\]
		Using (\ref{A''}) we obtain the relationship 
		$
		16a^2 + \kappa^2(c_3 - c_2) =  4ac_1$.
	\end{proof}
	Let us illustrate an example of this family. 
	\begin{example}
		Let $a = -\delta^2\kappa^2$ and choose $c_1 = 2\kappa^2$, $c_2 = - \kappa^2$ and $c_3 = - 8\kappa^2\delta^2 - 17\kappa^2$. To have for every choice of $\delta$ and $\epsilon$ solutions, let $A'(x)$ and $B'(y)$ be constant functions. To this end,
		\[
		A'(x) = 1 \quad B'(y) = \sqrt{4 + \kappa^2\delta^2}.
		\]
		Then the solution turns into
		\[
		\sinh\left(\kappa w(x,y) \right)
		= \tan\left(\kappa \left(x + \sqrt{4 + \delta^2}y \right) \right).
		\]
		Now, let us demonstrate some examples by selecting $\delta$ and $\kappa$. For $\kappa = \delta =  1$ then we obtain a solution of the hyperbolic sinh-Gordon equation
		\[
		\sinh(w(x,y)) = 
		\tan
		\left(
		x + \sqrt{5}y
		\right).
		\]
		We demonstrate its graph in Figure 3.
		\begin{figure}\caption{}
			\centering
			\includegraphics[scale = 0.6]{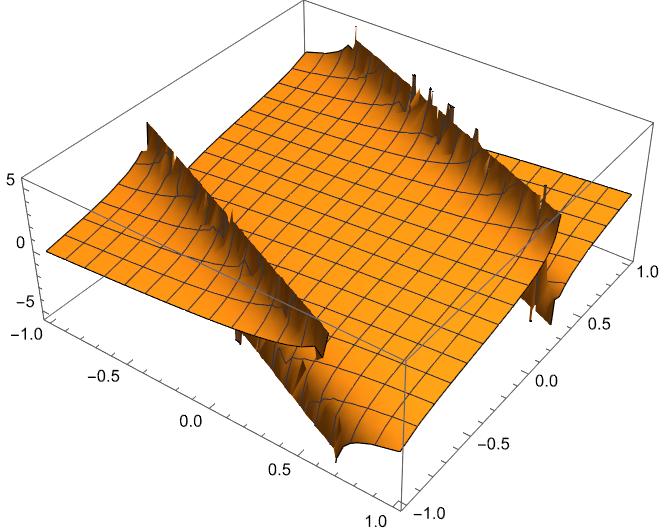}
		\end{figure}
		For $\kappa = i$ and $\delta = 1$, we obtain a solution for the sine-Gordon equation
		\[
		w(x,y) = \arcsin\left(\tanh\left(
		x + \sqrt{5}y
		\right)\right).
		\]
		We demonstrate its graph in Figure 4.
		\begin{figure}\caption{}
			\centering
			\includegraphics[scale = 0.6]{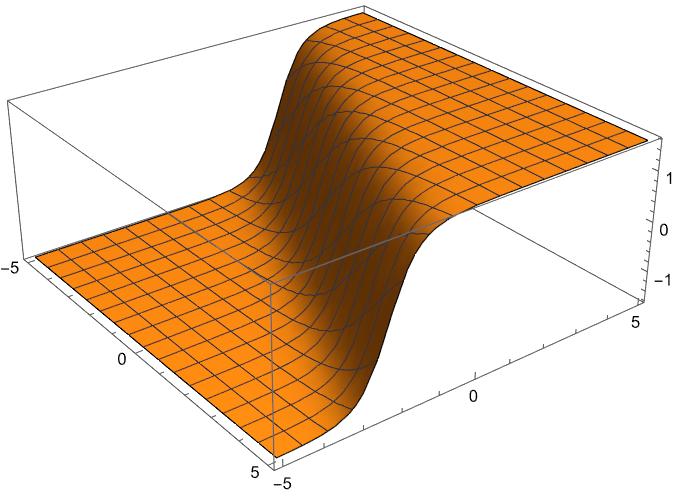}
		\end{figure}
	\end{example}

    	\section{Discussion}
	In this work, we classify all the hyperbolic/elliptic PDEs that admit functional separation solutions. Moreover, we study the hyperbolic/elliptic sine-Gordon and sinh-Gordon equations without any distinction between them. We have used the method of functional separation to obtain new families of solutions for equation (\ref{GShGordon}) and we prove that these are the solutions with such a property, so we obtain a classification. A future work to consider is to study this equation with a variable coefficient, \cite{YangZhong, Yang}, the inverse scattering method, \cite{AblowitzSegur} and the Fokas' method \cite{FokPell, Hwang} for boundary problems of both the hyperbolic and elliptic sine-Gordon and sinh-Gordon in a unified way.
	
	\textbf{Acknowledgements.} The author would like to thank Prof. C. Daskaloyannis and Prof. A. Fotiadis for their help and support and Dr. G. Papamikos for his useful comments.


	\vspace{10pt}
	\address{
		\noindent\textsc{Giannis Polychrou:}
		\href{mailto:ipolychr@math.auth.gr}
		{ipolychr@math.auth.gr}\\
		Department of Mathematics, Aristotle University of Thessaloniki,
		Thessaloniki 54124, Greece}
	
\end{document}